\documentclass{amsart}
\usepackage{amssymb, enumitem, hyperref, cleveref}
\newtheorem{prop}{Proposition}

\newtheorem{theorem}[prop]{Theorem}
\newtheorem{cor}[prop]{Corollary}
\theoremstyle{definition}
\newtheorem*{mydef}{Definition}
\DeclareMathOperator{\Spec}{Spec}

\newcommand{\U}{\mathcal{U}}
\setlist[enumerate,1]{label={\upshape(\arabic*)}}

\begin{document}
\title{Closed points on schemes}
\author{Justin Chen}
\address{Department of Mathematics, University of California, Berkeley,
California, 94720 U.S.A}
\email{jchen@math.berkeley.edu}

\subjclass[2010]{{14A25, 54F65}}

\begin{abstract}
This brief note gives a survey on results relating to existence of closed points on 
schemes, including an elementary topological characterization of the schemes with 
(at least one) closed point. 
\end{abstract}

\vspace*{-1.3cm}
\maketitle

Let $X$ be a topological space. For a subset $S \subseteq X$, let $ \overline{S} = \overline{S}^X$ 
denote the closure of $S$ in $X$. Recall that a topological space is \textit{sober} if every 
irreducible closed subset has a unique generic point. The following is well-known:

\begin{prop} \label{noetherianSoberSpace}
Let $X$ be a Noetherian sober topological space, and $x \in X$. Then $\overline{\{x\}}$ contains a 
closed point of $X$. 
\end{prop}
\begin{proof}
If $\overline{\{x\}} = \{x\}$ then $x$ is a closed point. Otherwise there exists 
$x_1 \in \overline{\{x\}} \setminus \{x\}$, so $\overline{\{x\}} \supseteq \overline{\{x_1\}}$.
If $x_1$ is not a closed point, then continuing in this way gives a descending chain of closed
subsets $\overline{\{x\}} \supseteq \overline{\{x_1\}} \supseteq \overline{\{x_2\}} \supseteq \ldots$
which stabilizes to a closed subset $Y$ since $X$ is Noetherian. Then $Y$ is the closure of 
any of its points, i.e. every point of $Y$ is generic, so $Y$ is irreducible. Since $X$ is sober, 
$Y$ is a singleton consisting of a closed point.
\end{proof}

Since schemes are sober, this shows in particular that any scheme whose underlying 
topological space is Noetherian (e.g. any Noetherian scheme) has a closed point. 

In general, it is of basic importance to know that a scheme has closed points (or not). For instance,
recall that every affine scheme has a closed point (indeed, this is equivalent to the axiom of choice).

In this direction, one can give a simple topological characterization of the schemes with closed 
points. First, some notions of redundancy for an open cover:

\begin{mydef}
Let $X$ be a topological space, and $\U$ an open cover of $X$. 
\begin{enumerate}[label={\upshape(\alph*)}]
\item $\U$ is said to be
\textit{totally redundant} if $\displaystyle U \subseteq \bigcup_{V \in \, \U \setminus \{U\}} V$ 
for every $U \in \U$. 

\item Let $P$ be a property of topological spaces. $\U$ is said to be \textit{P-redundant}
if $\displaystyle U \subseteq \bigcup_{V \in \, \U \setminus \{U\}} V$ for every $U \in \U$ with 
property $P$.
\end{enumerate}
\end{mydef}

Recall from \cite{H} that being an affine scheme is a topological property.

\begin{theorem} \label{closedPtsChar}
The following are equivalent for a scheme $X$:

\begin{enumerate}

\item $X$ has a closed point

\item There exists an open cover of $X$ that is not affine-redundant

\item There exists an open affine cover of $X$ that is not totally redundant.

\end{enumerate}
\end{theorem}
\begin{proof}
(3) $\implies$ (2): clear from definitions. 

(1) $\implies$ (3): Let $x \in X$ be a closed point. Then $X \setminus \{x\}$ is open, so let $\mathcal{V}$
be an open affine cover of $X \setminus \{x\}$ (recall that open affines form a basis for the topology of 
$X$). Choose an open affine neighborhood $U$ of $x$. Then
$\U := \mathcal{V} \cup \{U\}$ is an open affine cover of $X$ that is not totally redundant, since 
$\displaystyle x \in U \setminus \bigcup_{V \in \mathcal{V}} V$.

(2) $\Rightarrow$ (1): Let $\U$ be an open cover of $X$ that is not affine-redundant. For $U \in \U$ affine
with $\displaystyle U \not \subseteq \bigcup_{V \in \, \U \setminus \{U\}} V$, set $\displaystyle Z 
:= X \setminus \bigcup_{V \in \, \U \setminus \{U\}} V = U \setminus \bigcup_{V \in \, \U \setminus \{U\}} V$. 
Then $Z$ is a nonempty closed subset of $X$, so for any $x \in Z$, $\overline{\{x\}}^X = \overline{\{x\}}^Z$. 
Also $Z$ is a closed subset of 
$U$, hence $Z$ is an affine scheme (e.g. with the reduced induced subscheme structure). 
Thus $Z$ has a closed point, which is then also a closed point of $X$.
\end{proof}

Another way of stating \Cref{closedPtsChar} is that a scheme has no closed points iff every open 
cover is affine-redundant iff every open affine cover is totally redundant. Such schemes do exist, cf. 
\cite{O} for a quasi-affine example and \cite{S} for a general gluing construction.

The following is an immediate consequence of \Cref{closedPtsChar}:

\begin{cor} \label{qcClosed}
Let $X$ be a quasicompact scheme. Then $X$ has a closed point.
\end{cor}
\begin{proof}
If $X$ is quasicompact, then $X$ has a finite cover by open affines, and any finite cover has 
a subcover that is not totally redundant.
\end{proof}

Since a closed subset of a quasicompact space is quasicompact, \Cref{qcClosed} implies
that every nonempty closed subscheme of a quasicompact scheme $X$ contains a 
closed point of $X$. This leads to a characterization of the quasicompact schemes with as 
few closed points as possible:

\begin{prop}
The following are equivalent for a scheme $X$:

\begin{enumerate}

\item $X$ is quasicompact and $X$ has a unique closed point 

\item $X = \Spec R$ for some local ring $R$

\end{enumerate}
\end{prop}

\begin{proof}
(2) $\implies$ (1): an affine scheme is quasicompact. 
(1) $\implies$ (2): let $x$ be the unique closed point, and 
$U$ an open affine neighborhood of $x$. Then $X \setminus U$ is closed,
so if $X \setminus U$ were nonempty then it would contain a closed point
of $X$, contradicting uniqueness of $x$. Thus $X = U$ is affine.
\end{proof}

In particular, spectra of local rings can be distinguished among all schemes
purely by their topology.

Finally, as an extension of \Cref{noetherianSoberSpace}, it is also true
that every locally Noetherian scheme has a closed point (cf. 
\cite[\href{http://stacks.math.columbia.edu/tag/02IL}{Lemma 02IL}]{St}). 
Thus \Cref{closedPtsChar} implies that any locally Noetherian 
scheme has an open cover by spectra of Noetherian rings that is not totally
redundant (although this is not so obvious from the definitions).

\vskip 2ex

\end{document}